%% file: fillingreduced.tex
\documentclass[12pt]{amsart}
\usepackage{amsmath}
\usepackage{amssymb}
\usepackage{amsfonts}
\usepackage{amsthm}
\usepackage{verbatim}
\usepackage{cite}
\usepackage{leftidx}
\usepackage{enumerate}
\usepackage{txfonts}
\usepackage{manfnt}
\usepackage{amscd}
\usepackage{mathpazo}
\usepackage[mathscr]{eucal}
\usepackage{xr}
\usepackage{hyperref}
\input{myheader2-short.tex}
\title { Fibres of embedding dimension 2\\ in generic projections }
\normalsize
\author
{Ziv Ran}

\renewcommand{\theequation}
{\arabic{equation}}
\begin{document}

\pagestyle{plain}

\date {\today}
\address{\tiny  {\newline Ziv Ran \newline University of California
Mathematics Department\newline Surge Facility, Big Springs Road,
\newline Riverside CA 92521 US
\newline ziv.ran @ ucr.edu}}

\begin{abstract}
We consider a general fibre of given length in a generic projection of a variety.
Under the assumption that the fibre is of local embedding dimension 2 or less, an assumption
which can be checked in many cases, we prove that the fibre is reduced and its image
on the projected variety is
an ordinary multiple point. 
\end{abstract}
\maketitle
The study of generic linear projections of a smooth projective 
variety $X\subset\P^N$, especially their singularities,
 was a favorite topic of classical 
projective geometers for generations (see e.g. \cite{zak}), and these projections
remain
among the main sources of finite, non-flat morphisms in Algebraic Geometry.
Modern interest in this area seems to begin with Mather \cite{mather},
whose work was reinterpreted in algebro-geometric language by Alzati- Ottaviani
\cite{alzati-ottaviani}. One may formulate a folklore 'generic projection conjecture'
to the effect that the projection
of $X$ from a generic linear subspace $\Lambda\in\P^N$ has only the 'expected' singularities.
After numerous 
partial results including \cite{(n+2)sec}, \cite{beis}, 
the conjecture was recently proven for projections \emph{from
a point} by Gruson and Peskine \cite{grp} (which the reader
may consult  for further
introductory comments,   references, as well as examples and applications).\par
In \cite{filling}, we gave various extensions of the Gruson-Peskine Theorem, including
one about projections from \emph{lines}. More generally, we showed inter alia that the locus of
fibres of given length $k$ of a generic projection from a center $\Lambda$ of any dimension
$\lambda$ is smooth of the expected dimension, just as long as the fibres have local
embedding dimension 2 or less. The latter hypothesis is of course automatic for
$\lambda=0, 1$, but as we shall see below (cf. Lemma \ref{edim} and Example \ref{emb-example}), it
also holds in a fair range of cases beyond that. \par
One question which was left open in \cite{filling}
was as to the nature of the general fibre of given length $k$ and embedding dimension $\leq 2$: 
was it in fact in fact reduced,
i.e. a collection of distinct points; even more optimistically, are those in general position
relative to the $(\lambda+1)$-plane they are constrained to lie on ?
The purpose of this paper is to settle these questions. We will settle the first question
affirmatively, and show further that the resulting $k$-fold point on the projected image is
'ordinary' or 'transverse' (see below for definitions). The second question will be settled
negatively by a counterexample (see Example \ref{trisec}). Such counterexamples
seem rare though, and the possibility remains that
they could be effectively characterized in terms of low-degree
defining equations, though we don't have a specific conjecture- much less, result, to offer.
\par We begin with some notation. Let \[X\subset \P^m\] be an irreducible  
closed subvariety of codimension $c$. Let \[\Lambda=\P^\lambda\subset\P^m, \lambda<c\]
be a generic linear space. Let \[X^\Lambda_k\subset\P^{m-\lambda-1}\] 
be the (locally closed)
locus of $(\lambda+1)$-planes $L$ disjoint from the singular locus of $X$, 
containing $\Lambda$ and meeting $X$ in a scheme of
length $k$ or more.  
Now consider the following condition on $L\in X^\Lambda_k$:
\nl\nl

(EDIM2) The local embedding dimension of $Z=L\cap X$ is 2 or less everywhere.
\nl\nl

 Under this condition, it was proven in
\cite{filling} that $X^\Lambda_k$ is smooth of the expected codimension in $\P^{m-\lambda-1}$, viz.
$k(c-\lambda-1)$, at $L$. The tangent space to $X^\Lambda_k$ at a point $L$
can be identified with $H^0(N^\rms)$ where $N^\rms$ is a certain subsheaf
of the normal bundle $N_{L/\P^m}$ which differs from it only at points of $Z=L\cap X$,
and at those points consists grosso modo of the vectors normal to $L$ and tangent to $X$.\par
Before stating the new result, we make a remark about the EDIM2 condition.
\begin{lem}\label{edim}
Notations as above, condition {\rm{(EDIM2)}} holds for all $L\in X^\Lambda_k$ provided either\par
(i) $\min(\lambda+1, m-c)\leq 2$; or\par
(ii) $m<4c-3\lambda+6$.
\end{lem}
\begin{proof}
(i) is trivial. As for (ii), assuming $m-c\geq 3, \lambda>1$, 
note that if $Z$ has embedding dimension 3 or more at $p$, then
the embedded tangent space $\hat T_pX$ meets $L$ in at least a 3 (projective)-  dimensional
space, hence meets $\Lambda$, a hyperplane in $L$, in at least a 2-  dimensional
space. The space of pairs $(p, U)$, where $U$ is a 2-dimensional subspace of
$\hat T_pX$, is $4(m-c)-6$-dimensional, and it is $3(m-\lambda)$ conditions for $U$ to be contained
in $\Lambda$. Since $\Lambda$ is general with respect to $X$, 
the Bertini theorem on transversality of a general translate 
yields the result.
\end{proof}

As a matter of terminology, a collection of subspaces $A_i$ of a vector space $B$ is said to be
\emph{transverse} if
\[\codim(\bigcap A_i, B)=\sum \codim(A_i, B).\]
A point $y$ on a subvariety $Y$ in a smooth variety $P$ is said to be a \emph{transverse}
$k$-fold point if locally at $y$, $Y$ is a union of $k$ smooth branches $Y_1,...,Y_k$ and
$T_yY_i, i=1,...,k$ is a transverse collection of subspaces of $T_yP$.


\begin{thm}\label{main}
Notations as above, assume\par
(a)  $L$ is general in $X^\Lambda_k$; 
\par (b) $c>\lambda+1$; \par
(c) condition {\rm{ (EDIM2)}} holds.\nl Then \par
(i) $Z$ is reduced;\par
(ii) each point of $Z$ is general on $X$;\par
(iii) $Z$ projects to a transverse $k$-fold point on the projection $\pi_\Lambda(X)$.
\end{thm}
Note that in (ii), it is not claimed that $Z$ is a generic $k$-tuple on $X$, only
each point in itself is generic. Thus given any subvariety $Y\subsetneqq X$, $Z$ may
be assumed disjoint from $Y$ by taking $\Lambda$ general enough.
\begin{example}\label{emb-example}
For $m-c=3$, i.e. $X$ a 3-fold, and $\lambda=2$, Lemma \ref{edim} shows
that condition (EDIM2) is weaker than $c>\lambda+1=3$.
Therefore the theorem is applicable to the generic projection of any 
smooth 3-fold  $X\subset\P^m, m\geq 7$,
to $\P^{m-3}$. In particular, for $m=7$, the projected image of $X$
in $\P^4$ has finitely many transverse 4-fold points, each with
4 linearly independent tangent hyperplanes. \par
In fact, for $\lambda=2$, i.e. projection from a plane, and $X$ of codimension $c=4$, the theorem
applies for $X$ of dimension up to 11, i.e. for projection from $\P^m$ to $\P^{m-3}$ for $m\leq 15$.\qed
\end{example}

The proof of the Theorem is a continuation of the argument used to prove Theorem 5.1 in \cite{filling}.
A key role is played again by the vanishing lemma (\cite{filling}, Lemma 5.8) 
for subsheaves of the normal bundle $N_L$ whose sections move with $\Lambda$ (which,
we recall, is a generic $\lambda$-plane). Originally, the Lemma was applied to the secant sheaf
$N^\rms$. Here, we note that it can be applied as well  to suitable subsheaves of $N^\rms$,
and exploit the consequences. Incidentally, we neglected to mention in \cite{filling}
that the vanishing lemma implies a regularity result of the secant sheaves, and this omission 
will be rectified below.
\begin{proof}[Proof of Theorem]
We fix a sufficiently general $L$ as in the Theorem. Then by \cite{filling}, Theorem 5.1,
we may assume that $Z=L\cap X$ has
length exactly $k$. Pick any $p\in Z$. \par
Now recall the 'secant subsheaf' $N^\rms\subset\hom(\I_L, \O_L)$ defined in
\cite{filling}, which fits in an exact diagram
\eqspl{tp}{
\begin{matrix}
0&&0&&\\
\downarrow&&\downarrow&&\\
N^\rms&\to&\hom(\I_Z/\I_X, \O_Z)&\to&0\\
\downarrow&&\downarrow&&\\
\hom(\I_L, \O_L)&\to&\hom(\I_L. \O_Z)&\to&0\\
\downarrow&&\downarrow&&\\
\hom(\I_L\cap\I_X, \O_Z)&=&\hom(\I_L\cap\I_X, \O_Z)&&\\
\end{matrix}
}
(note $\I_Z/\I_X=(\I_L+\I_X)/\I_X=\I_L/(\I_L\cap\I_X)$).
Of course, $\hom(\I_Z/\I_X, \O_Z)$ is just a skyscraper sheaf with stalk
equal to the tangent space at $Z$ to the punctual Hilbert scheme of $X$, that is
$X\sbr k.$
where $k=\l(Z)$. Now we
have a support decomposition
\[\hom(\I_Z/I_X, \O_Z)=\bigoplus T_p. T_p=\hom(\I_Z/I_X, \O_Z)_p\]
which corresponds to a local analytic decomposition
\[X\sbr k.=\prod X\sbr k(p)., k(p):=\l_p(Z).\]
via the identification of Zariski tangent spaces
\eqspl{}{T_{Z_p}X\sbr k(p).=H^0(T_p).
} It was observed in \cite{filling}- and of course is well known- that the
assumption of embedding dimension 2 or less implies, via Fogarty's theoerm,
that $X\sbr k(p).$ is smooth at $Z_p$ and its general point corresponds
to a reduced scheme.\par
Now I claim that the surjection
\[N^\rms\to T_p\]
(cf. \eqref{tp}) induces a surjection on $H^0$:
\eqspl{loc}{H^0(N^\rms)\to H^0(T_p).
}
 If this holds, then the natural (local) map
from the secant scheme $X^\Lambda_k$ 
to $X\sbr r(p).$ is smooth, hence surjective, near $Z$. Therefore $Z_p$ is general on $X\sbr r(p).$,
hence reduced. Since $Z_p$ is supported at $p$,
it follows that $Z_p=p$ is a single 
reduced general point of $X$, this proving parts (i) and (ii) of the Theorem.\par
To prove surjectivity of \eqref{loc}, note 
that its cokernel is anyhow of the form $H^0(A)$ for some quotient $A$
of the skyscraper sheaf $T_p$, such that
the natural map $H^0(N^\rms)\to H^0(A)$ vanishes. 
Let $N^\rms_0\subset N^\rms$ be the corresponding subsheaf of $N^\rms$
under the surjection $N^\rms\to T_p$.
Thus we have an exact sequence
\eqspl{a}{\exseq{N^\rms_0}{N^\rms}{A}.
}
Then by the generality of $L$, $N^\rms_0$ 
satisfies the hypotheses of Lemma 5.8
in \cite{filling}, and hence $H^1(N^\rms_0)=0$. But then the exact sequence
\eqref{a}
yields surjectivity of $H^0(N^\rms)\to H^0(A)$, which is a contradiction unless $A=0$.\par
It remains to prove assertion (iii). 
To this end consider the exact sequence
\eqspl{quot}{
\exseq{N^\rms(-1)}{N_L(-1)}{\bigoplus\limits_{p\in Z}Q_p}
}
where $Q_p$ is a skyscraper $k(p)$-module of length $c-\lambda-1$ which, writing 
\[\P^m=\P(V^*), L=\P(U^*), U\subset V,\] 
we may identify
with the vector space
\[H^0(Q_p)=V/(U+T_p(X)(-1)).\]
By \cite{filling}, Lemma 5.8, we have $H^1(N^\rms(-1))=0$
, whence surjectivity of
\[q: H^0(N_L(-1))=\Hom(U, V/U)\to \bigoplus H^0(Q_p).\]
On the other hand, as $N_L(-1)$
 is a trivial bundle on $L\simeq\P^{\lambda+1}$,  it is easy to see from \eqref{quot}
that $H^i(N^\rms(-1))=0, i>1$, and consequently
\eqspl{chi}{
\chi(N^\rms(-1))=h^0(N^\rms(-1))=(m-\lambda-1)-k(c-\lambda-1).
}
Let $W_i$ be the $(m-c)$-dimensional , $(c-\lambda-1)$-codimensional 
subspace of $V/U$ corresponding to
$T_{p_i}X$. Note that each $W_i$ projects to the tangent space of the branch
of $\pi_\Lambda(X)$ corresponding to $p_i$. By \eqref{quot}, $\bigcap W_i=H^0(N^\rms(-1))$.
Therefore \eqref{chi} proves our transversality assertion.
 \end{proof}
 It has been noted above that $H^i(N^\rms(-i))=0,  \forall i>1$. 
 Together with the aforementioned Lemma 5.8 from
 \cite{filling} and the Castelnuovo-Mumford Lemma (\cite{mumford-curves-surface}, p. 99) , we conclude
 \begin{cor}
 The sheaf $N^\rms$ is 0-regular. In particular, $N^\rms$ is globally generated.\qed
 \end{cor}
\begin{rem}  The general fibre of given length (e.g. 3) of generic projection from a line need not be
  in general position relative to the plane containing it. 
  The following is an example of a smooth nondegenerate surface $X$ in 
  $\P^5$ admitting a 3-parameter family $\{D_t\}$ of trisecant lines filling a hypersurface.
  Then, a generic line $\Lambda$ will meet finitely many of the $D_t$ and the corresponding
  aligned triples $D_t\cap X$ will be fibres of projection from $\Lambda$. Our vague impression-
  unbacked by evidence- 
  is that such examples are  not easy to construct and may be 'special' in some characterizable way
  related to the defining equations.
  \end{rem}
 
 \begin{example}\label{trisec} Let $F=(f_{ij})$ be a general symmetric $4\times 4$  matrix of 
 linear forms on $\P^5$.
 Let $X$ be the locus $\rk(F)<3$, i.e. the  
 zero-locus of the $3\times 3$ minors of $F$, which is a smooth surface
 cut out by cubics;  let $Y$ be the locus $\rk(F)<4$, i.e. the
  zero-locus of $\det(F)$, which is a quartic
 singular along $X$. 
 More invariantly, $F$ corresponds to a quadratic form $q$ on the trivial bundle $4\O_{\P^5}$ with
 values in $\O_{\P^5}(1)$, and $Y$ and $X$ are respectively the loci where $q$ drops rank
 by 1 or 2 at least.
 For $t=(t_1,...,t_4)$ general, let $D_t$ be the line 
 in $\P^5$ with equations
 $\sum\limits_i t_if_{ij}, j=1,...,4$. The lines $D_t$ sweep out $Y$. I claim that a general
 $D_t$ is trisecant to $X$. To see this, we can assume by a suitable change of notations that $D_t$
 corresponds to the 1st row (= 1st column) of $F$. Then the only $3\times 3$ minor
 nontrivial on $D_t$ is the one corresponds to rows and columns $2,3,4$, which gives
 a single cubic equation on $D_t$ defining the intersection $D_t\cap X$.\qed
 \end{example}
 \subsection*{Open questions}
 In conclusion, we point out 2 important questions that remain open.\par 
 (i) Let $Z$ be a general fibre of given length in a generic projection of a smooth variety.
 Then is $Z$ smoothable ? (The question only arises if $Z$ has embedding dimension $>2$.)
 
 (ii) Describe the singularity of $X^\Lambda_k$ at a point $L$ such that $L\cap X$
 has length $>k$ and embedding dimension 2 (the case of embedding dimension $\leq 1$
 was settled in \cite{filling}).
\vskip1in
\bibliographystyle{amsplain}
\bibliography{../mybib}
\end{document}

%% file: myheader2-short.tex
\DeclareMathOperator{\codim}{codim}

\def\rms{\mathrm{s}}

 \DeclareMathOperator{\Hom}{Hom}

\def\refp #1.{(\ref{#1})}

\def\sbr #1.{^{[#1]}}
\def\sfl #1.{^{\lfloor #1\rfloor}}

\def\?{{\bf{??}}}

\def\P{\mathbb P}

\def\O{\mathcal O}

\def\rk{\text{rk}}

\def\g{\mathfrak g}

\def\hom{\mathfrak {hom}}

\def\1/2{\frac{1}{2}}

\def\I{\mathcal{ I}}

\def\2{{[2]}}
\def\l{\ell}
\def\nl{\newline}

\def\hom{\mathcal{H}\mathit{om}}

\def\<{\langle}
\def\>{\rangle}

\def\2{{[2]}}
\def\l{\ell}

\def\scl #1.{^{\lceil#1\rceil}}
\def\spr #1.{^{(#1)}}
\def\sbc #1.{^{\{#1\}}}

\def\subpr#1.{_{(#1)}}

\renewcommand{\theequation}{\arabic{section}.\arabic{subsection}.\arabic{equation}}
\def\beq{\begin{equation*}}
\def\eeq{\end{equation*}}

\def\g3{{\Gamma\spr 3.}}

\newcommand{\eqspl}[2]{
\begin{equation}\label{#1}
\begin{split}
#2\end{split}\end{equation}}

\newcommand{\exseq}[3]{
0\to #1\to #2\to #3\to 0
}
\newcommand{\beginalphaenum}{
\begin{enumerate}\renewcommand{\labelenumi}{ }
\item \begin{enumerate}
}

\def\eex{\end{rm}\end{example}}

\newtheorem{thm}{Theorem}

\newtheorem*{thm*}{Theorem}
\newtheorem*{prop*}{Proposition}
\newtheorem{cor}[thm]{Corollary}
\newtheorem*{cor*}{Corollary}

\newtheorem{lem}[thm]{Lemma}
\newtheorem{lem*}{Lemma}

\newtheorem*{claim*}{Claim}

\theoremstyle{remark}

\newtheorem{rem}[thm]{Remark}
\newtheorem{crit-rem}[thm]{Critical remark}

\newtheorem{example}[thm]{Example}
\newtheorem*{example*}{Example}

\newtheorem*{defn*}{Definition}